\documentclass{amsart}
\usepackage{amsmath,amssymb,amsthm,bm}
\bmdefine{\be}{e}
\bmdefine{\bg}{g}
\bmdefine{\bk}{k}
\bmdefine{\bm}{m}
\bmdefine{\bp}{p}
\bmdefine{\bs}{s}
\bmdefine{\bt}{t}
\bmdefine{\bw}{w}

\newtheorem{thm}{THEOREM}[section]
\newtheorem{lem}[thm]{LEMMA}

\newtheorem{prop}[thm]{PROPOSITION}

\newcommand{\Romannum}[1]{\uppercase\expandafter{\romannumeral#1}}
\makeatletter
  
  \@addtoreset{equation}{section}
\makeatother
\begin{document}
\title[Subwords in base-\(q\) expansions]{On subwords in the base-\(q\) expansion of polynomial and exponential functions}
\author{Hajime Kaneko}
\address{Institute of Mathematics, University of Tsukuba, 1-1-1, Tennodai, Tsukuba, Ibaraki, 305-8571, JAPAN; 
Center for Integrated Research in Fundamental Science and Engineering (CiRfSE), University of Tsukuba, 1-1-1, Tennodai, Tsukuba, Ibaraki 305-8571, JAPAN}
\email{kanekoha@math.tsukuba.ac.jp}
\author{Thomas Stoll}
\address{1. Universit\'e de Lorraine, Institut Elie Cartan de Lorraine, UMR 7502, Vandoeuvre-l\`es-Nancy, F-54506, France;
2. CNRS, Institut Elie Cartan de Lorraine, UMR 7502, Vandoeuvre-l\`es-Nancy, F-54506, France}
\email{thomas.stoll@univ-lorraine.fr}
\begin{abstract}
  Let $w$ be any word over the alphabet $\{0,1,\ldots, q-1\}$, and denote by $h$ either a polynomial of degree $d\geq 1$ or $h: n\mapsto m^n$ for a fixed $m$. Furthermore, denote by $e_q(w;h(n))$ the number of occurrences of $w$ as a subword in the base-$q$ expansion of $h(n)$. We show that
\[
\limsup_{n\to\infty} \frac{e_q(w;h(n))}{\log n}\geq \frac{\gamma(w)}{l\log q},
\]
where $l$ is the length of $w$ and $\gamma(w)\geq 1$ is a constant depending on a property of circular shifts of $w$. This generalizes work by the second author as well as is related to a generalization of Lagarias of a problem of Erd\H os.
\end{abstract}
 \maketitle
\maketitle

\section{Introduction}

Let \(q\geq 2\) be an integer and \(w\) a nonempty finite word over the alphabet 
\(\mathcal{A}_q:=\{0,1,\ldots,q-1\}\). We denote by $l=l(w)$ the length of $w$ which is the number of symbols (or letters) in $w$. For any integer \(n\geq 1\), consider the finite base-$q$ expansion of $n$,
$$n=\sum_{i=0}^M n_iq^i,$$
where $M=M(n)=\lfloor \log_q n\rfloor$ denotes the position of the most significant digit. We write
\[(n)_q=n_Mn_{M-1}\cdots n_0\]
as a shorthand notation and regard this as a word over $\mathcal{A}_q$. For convenience, put \((0)_q:=0\). In this paper, we are concerned with the quantity \(e_q(w;n)\) which denotes the number of  (possibly overlapping) occurrences of the word \(w\) in the finite base-\(q\) expansion of $n$. For example, for $q=10$, $w=202$ and $n=20202$ we have $e_{10}(202;20202)=2$. In what follows, we denote by \(w^k\) the \(k\)-th concatenation power of a word \(w\);  if \(k=0\), then \(w^k\) will denote the empty word. For instance, for the word $w=20$ and $k=3$ we have $w^k=202020$.

The investigation on the number of occurrences of subwords in digital expansions along special subsequences of integers has undergone some fundamental progress in recent times. A classical point of view, dating back to the work of Gelfond~\cite{Ge67}, is to study the distribution in residue classes. The related sequences are automatic sequences such as, for example, the Thue--Morse sequence or the Rudin--Shapiro sequence. We refer the reader to~\cite{DMR17, Ha17, MR15, Mu17, St16} for an up-to-date list of the related work. 

A second and different problem is to investigate the number of occurrences of digital blocks in these rarefied sequences. We will consider this problem along polynomial and exponential subsequences in the present paper. We will show that for any fixed $w$ there are terms in these rarefied sequences whose base-$q$ expansion contains \textit{not too few} occurrences of $w$ as subwords. For that purpose we will establish lower bounds on the maximal order of magnitude of the associated counting function.

We denote the set of nonnegative integers (resp. positive integers) 
by \(\mathbb{N}\) (resp. \(\mathbb{Z}^{+}\)) and use the standard Landau resp. Vinogradov notation $f=O(g)$ resp. $f\ll g $ to indicate that $|f|\leq C |g|$ for some absolute constant $C>0$. As common, we denote a possible dependence on the parameters in the index of the symbols.

For a better understanding of the flavour of our results, let us first give two examples in the case of a polynomial rarification. 

First, consider $w'=0^l$ ($l$ fixed) which is the $l$-th concatenation power of the single letter $0$ and let \(f(X)\in \mathbb{Z}[X]\) be any arbitrary but fixed  polynomial of degree \(d\geq 1\) with \(f(\mathbb{N})\subset \mathbb{N}\). Since $e_q(w;f(n))\leq \frac{\log f(n)}{\log q}$ for sufficiently large $n$ we have
$$\limsup_{n\to\infty} \frac{e_q(w;f(n))}{\log n}\leq \frac{d}{\log q}.$$
On the other hand, by choosing a positive integer \(a\) such that the coefficients of \(f(X+a)\) are all positive, we have 
\[\limsup_{n\to\infty} \frac{e_q(w;f(n))}{\log n}
\geq \limsup_{L\to\infty} \frac{e_q(w;f(q^L+a))}{\log (q^L+a)}\geq \frac{d}{\log q}.\]
In fact, in the base-$q$ expansion of $f(q^L+a)$ the $d$ blocks of 0's between consecutive powers of $q$ are each of length $L+O_{q,f}(1)$ as $L\rightarrow \infty$. This leads to 
\begin{equation}\label{limsup1}
  \limsup_{n\to\infty} \frac{e_q(w;f(n))}{\log n}= \frac{d}{\log q}.
\end{equation}

As a second example, on the other end of the spectrum, let $w'=(q-1)^l$ be the $l$-th concatenation power of the single letter $q-1$. Theorem 1 in~\cite{St16} states that there exists $N_0(q, f, l) > 1$ such that for all $N \geq N_0(q, f, l)$ there is an $n$ with $e_q(w';f(n)) = N.$ From the method of the proof, it follows that
\begin{equation}\label{limsup2}
  \limsup_{n\to\infty} \frac{e_q(w';f(n))}{\log n}\geq \frac{1}{\log q}.
\end{equation}
In fact, in the proof the author generates \textit{one} block of consecutive $q-1$'s, hence also losing the factor $d$ with respect to the previous result.

We conjecture that~(\ref{limsup1}) holds true for any $w$, however, this seems to be a very difficult question. 

Our first result gives a result for general $w$ in the spirit of~(\ref{limsup2}) and deals with a question posed in~\cite{St16}. Denote by $\gamma'(w)$ the number of occurrences of $w$ in $w^2$ (circular shifts) and put $\gamma(w)=\gamma'(w)-1$ (for example, $\gamma(2020)=2$, $\gamma(0^l)=\gamma((q-1)^l)=l$). Note that $1\leq \gamma(w)\leq l$ for all non-empty words $w$. 

\begin{thm}
Let \(f(X)\in \mathbb{Z}[X]\) be a polynomial of degree \(d\geq 1\) with \(f(\mathbb{N})\subset \mathbb{N}\). 
Let \(w\) be a word over the alphabet \(\mathcal{A}_q\) with length \(l\geq 1\). Then 
\[
\limsup_{n\to\infty} \frac{e_q(w;f(n))}{\log n}\geq \frac{\gamma(w)}{l\log q}.
\]
\label{thm1}
\end{thm}

Our second result concerns exponential functions. A famous (still open) problem by Erd\H os says that for all sufficiently large $n$ the ternary expansion of $2^n$ always contains the digit 2. We refer to the article of Lagarias~\cite{lag} and to~\cite{DW16} for recent and related results. Lagarias~\cite[Conjecture 1.12]{lag} generalized Erd\H os' conjecture: For all multiplicatively independent positive integers $m$ and $q$ the base-$q$ expansion of the integers $m^n$, $n=1,2\ldots$ contain any given word $w$ in its base-$q$ expansion for all sufficiently large $n\geq n_0(w)$. While Theorem~\ref{thm2} does not provide an answer to this conjecture it gives a quantitative lower bound along a subsequence of integers and therefore (up to a constant factor) the correct maximal order of magnitude.

\begin{thm}
Let $p$ be a prime number, \(m\) be a positive integer not a power of \(p\) and \(w\) a finite word over the alphabet
\(\mathcal{A}_p\) with length \(l\geq 1\). Then we have 
\[
\limsup_{n\to\infty} \frac{e_p(w;m^n)}{\log n}\geq \frac{\gamma(w)}{l\log p}.
\]
\label{thm2}
\end{thm}

In Section~\ref{sec1} we provide a proof of Theorem~\ref{thm1} and Section~\ref{sec2} is devoted to a proof of Theorem~\ref{thm2}. Both proofs are based on Hensel's lifting lemma. For a prime number \(p\) we use \(\mathbb{Z}_p\) for the ring of \(p\)-adic integers and \(\mathbb{Q}_p\) for the field of \(p\)-adic numbers; we denote by \(v_p(u)\) the \(p\)-adic order of \(u\in\mathbb{Z}_p\). 

\section{Proof of Theorem~\ref{thm1}}\label{sec1}

In what follows, we suppose that \(w\ne 0^l\) since we have a better result by~(\ref{limsup1}) in the case of a block consisting of $0$'s only.
We choose \(a_0\) to be a nonnegative integer satisfying \(f'(a_0)\ne 0\). We write 
\[w=0^k w_{k+1}\cdots w_l,\qquad k+1\leq l\] with \(w_{k+1}\ne 0\),
where all of the $w_i$, $i=k+1,\ldots, l$ are of length 1 (letters). 
\begin{lem}
There exists a nonnegative integer \(c=c(q,f)\), depending only on \(q\) and \(f(X)\), satisfying the following: 
For any positive integer \(L\), there exists a nonnegative integer \(N=N(q,f,L)\) such that 
the base-\(q\) expansion of \(f(N)\) is of the form
\[(f(N))_q=v w_{k+1}\cdots w_l w^{L-1} 0^c (f(a_0))_q,\]
where \(v\) is a finite word over \(\mathcal{A}_q\) or the empty word. 
\label{lem1}
\end{lem}
\begin{proof}
Let \(q:=p_1^{e_1}\cdots p_t^{e_t}\), where \(p_1,\ldots,p_t\) are distinct prime factors of \(q\) and 
\(e_1,\ldots,e_t\) are positive integers. Let \(b_{q,L}\) be a nonnegative integer whose base-\(q\) expansion is 
denoted as 
\[(b_{q,L})_q=w_{k+1}\cdots w_l w^{L-1} 0^c (f(a_0))_q,\]
for some $c$ that we will determine later. \par
Let \(L'\) be the length of the word \(w^L 0^c (f(a_0))_q\). 
For any \(i=1,\ldots,t\), consider the \(p_i\)-adic order of an integer \(m\) by \(v_{p_i}(m)\). 
If \(c\) is sufficiently large depending only on \(q\) and \(f(X)\), then we see for any \(i=1,\ldots,t\) that 
\[v_{p_i}(f(a_0)-b_{q,L})>2v_{p_i}(f'(a_0))\]
by \(f'(a_0)\ne 0\). Putting 
\[g(X):=f(X)-b_{q,L},\]
we get 
\[v_{p_i}(g(a_0))>2v_{p_i}(f'(a_0))=2v_{p_i}(g'(a_0)).\]
By Hensel's lifting lemma~\cite{lang} there exists a \(p_i\)-adic integer \(\xi_i\in\mathbb{Z}_{p_i}\) such that 
\(f(\xi_i)=b_{q,L}\). Thus, for any \(i=1,\ldots,t\), there exists an integer \(N_i\leq p_i^{L' e_i}\) such that 
\[f(N_i)\equiv b_{q,L} \pmod {p_i^{L' e_i}}.\]
By the Chinese remainder theorem, there is an integer $N$ with 
\begin{align}
0\leq N<p_1^{L' e_1}\cdots p_t^{L' e_t}=q^{L'}
\label{eqn1}
\end{align}
and 
\[N\equiv N_i \pmod {p_i^{L' e_i}}\]
for any \(i=1,\ldots,t\). Consequently, we obtain 
\[f(N)\equiv b_{q,L} \pmod {q^{L'}},\]
which implies the lemma. 
\end{proof}
In what follows, we use the integer \(N\) constructed in the proof of Lemma~\ref{lem1} (note that $N<q^{L'}$, see~(\ref{eqn1})). For any positive integer \(L\), we see by 
Lemma~\ref{lem1} that 
\begin{align}
e_q(w;f(N))\geq \gamma(w)(L-2).
\label{eqn2}
\end{align}
By~(\ref{eqn1}) and the definition of \(L'\), we get 
\begin{align}
N<q^{L'}\leq q^{l L+c'},
\label{eqn3}
\end{align}
where \(c'=c'(q,f)\) is a constant depending only on \(q\) and \(f(X)\). 
Thus, we obtain from~(\ref{eqn2}) and~(\ref{eqn3}) that 
\[\frac{1}{l\log q}-\frac{c'}{l\log N}\leq \frac{L}{\log N}
\leq \frac{2}{\log N}+\frac{e_q(w;f(N))}{\gamma(w)\log N}.\]
Noting that \(N\) tends to infinity as \(L\) tends to infinity (by \(w\ne 0^l\)), we deduce the theorem by the 
inequality above. This concludes the proof of Theorem~\ref{thm1}.

\section{Proof of Theorem~\ref{thm2}}\label{sec2}

For the proof of Theorem~\ref{thm2}, we first introduce a generalization of Hensel's lemma and define the notation which we use throughout this section. Let \(p\) be a prime number. 
For any positive integer \(m_1\) with \(m_1\equiv 1 \ (\mathrm{mod} \ p)\), we set \(m_1=1+ap^e\), where 
\(a,e\) are positive integers with \(p \nmid  a\). 
Put \(g(u):=(1+a p^e)^u\) for any \(u\in\mathbb{Z}_p\). 
Let again \(v_p(u)\) be the \(p\)-adic order of \(u\in\mathbb{Z}_p\). 
It is known that for any \(u,u'\in\mathbb{Z}_p\) with \(v_p(u-u')\geq N\) and \(N\in\mathbb{N}\), we have
\begin{align}
v_p\big(g(u)-g(u')\big)\geq N+1
\label{eqn:2-1}
\end{align}
(see~\cite[Chapter 2, p.26]{kob}). 

Let \(F\) be a function from \(\mathbb{Z}_p\) to \(\mathbb{Z}_p\) and let \(u\in \mathbb{Z}_p\), 
\(s\in\mathbb{Z}^{+}\), and \(N\in\mathbb{N}\). We call \(F\) differentiable modulo \(p^s\) at \(u\) 
with order \(N\) if there exists \(\partial_s F(u)\in\mathbb{Q}_p\) satisfying, for any integer \(k>N\) and 
\(h\in\mathbb{Z}_p\), 
\begin{align}
F(u+p^k h)\equiv F(u)+p^k h \partial_s F(u) \pmod  {p^{k+s}}.
\label{eqn:2-2}
\end{align}
Note that if we add a constant term to \(F\), then both the differentiability of $F$ and the value \(\partial_s F(u)\) 
are not changed. \par
In the following proposition we generalize the second statement of Corollary 2.6 in~\cite{axe}. This is needed in order to consider the case where the derivative is not a $p$-adic unit. We investigated this concept for general continuous functions that are not necessarily differentiable in~\cite{KS17}.
\begin{prop}
Let \(F\) be a function from \(\mathbb{Z}_p\) to \(\mathbb{Z}_p\). Let \(j,n,s,N\) be nonnegative integers 
with \(j+N<n\) and \(j<s\) and let \(u\in\mathbb{Z}_p\). Assume that 
\begin{align}
v_p\big(F(u)\big)\geq n.
\label{eqn:2-3}
\end{align}
Moreover, suppose for any \(x\in\mathbb{Z}_p\) with \(x\equiv u \ (\mathrm{mod} \ p^{n-j})\) that 
\(F\) is differentiable modulo \(p^s\) at \(x\) with order \(N\) and that 
\begin{align}
v_p\big(\partial_s F(x)\big)=j.
\label{eqn:2-4}
\end{align}
Then there exists a \(\xi\in\mathbb{Z}_p\) satisfying 
\begin{align*}
F(\xi)=0
\end{align*}
and
\begin{align*}
\xi\equiv u \pmod  {p^{n-j}}.
\end{align*}
\label{pro1}
\end{prop}
\begin{proof}
We construct \(\xi\in\mathbb{Z}_p\) satisfying the conditions of Proposition~\ref{pro1}, using the Newton method. 
It suffices to show that there exists a \(u_1\in\mathbb{Z}_p\) satisfying 
\begin{align}
v_p\big(F(u_1)\big)\geq n+1
\label{eqn:2-5}
\end{align}
and
\begin{align}
u_1\equiv u \pmod  {p^{n-j}}.
\label{eqn:2-6}
\end{align}
In fact, \(u_1\) will then satisfy~(\ref{eqn:2-3}), the assumption on the differentiability, and~(\ref{eqn:2-4}) with 
new nonnegative integers \(j_1=j\), \(n_1=n+1\), \(s_1=s\), and \(N_1=N\) because if \(x\in\mathbb{Z}_p\) satisfies 
\(x\equiv u_1 \ (\mathrm{mod} \ p^{n_1-j_1})\), then \(x\equiv u \ (\mathrm{mod} \ p^{n-j})\). \par
Let \(i\) be an integer with \(0\leq i\leq p-1\). Noting that \(n-j>N\) and \(n-j+s\geq n+1\), we see by~(\ref{eqn:2-2}) 
that 
\[
F\left(u+p^{n-j}\cdot i\right)
\equiv 
F(u)+p^{n-j}\cdot i \partial_s F(u) \pmod  {p^{n+1}}.
\]
Using 
\[v_p\left(p^{n-j}\partial_s F(u)\right)=n\leq v_p\big(F(u)\big),\]
we find \(i\) satisfying 
\[F(u+p^{n-j}\cdot i)\equiv 0 {\pmod  {p^{n+1}}}.\]
Putting \(u_1:=u+p^{n-j}\cdot i\), we obtain~(\ref{eqn:2-5}) and~(\ref{eqn:2-6}). 
\end{proof}
We now prove the differentiability of the function \(g(u)=(1+ap^e)^u\), where \(a\) and \(e\) are positive integers with 
\(p \nmid  a\). 
\begin{prop}
Let \(u\in\mathbb{Z}_p\). \\
\(\mathrm{1)}\) Suppose that \(e\geq 2\) or \(p\geq 3\). Then, for any \(u\in\mathbb{Z}_p\), we have that
\(g\) is differentiable modulo \(p^{e+1}\) at \(u\) with order \(0\). Moreover, 
\[\partial_{e+1} g(u)=a p^e.\]
\\
\(\mathrm{2)}\) Assume that \(e=1\) and \(p=2\). Let \(1+a' \cdot 2^t:=(1+2a)^2\), where 
\(a'\) and \(t\) are integers with \(2 \nmid  a'\) and \(t\geq 3\). Then  
\(g\) is differentiable modulo \(2^t\) at \(u\) with order \(0\). Moreover, 
\[\partial_{t} g(u)=a' 2^{t-1}.\]
\label{pro2}
\end{prop}
For the proof of Proposition~\ref{pro2}, we need the following auxiliary result. 
\begin{lem}
Assume that \(e\geq 2\) or \(p\geq 3\). 
Let \(k\) be a nonnegative integer and \(h\in\mathbb{Z}_p\). Then we have 
\begin{align}
(1+ap^e)^{hp^k}\equiv 1+ah p^{k+e} \pmod  {p^{k+e+1}}.
\label{eqn:2-7}
\end{align}
\label{lem2-1}
\end{lem}
\begin{proof}
We may assume that \(h\) is a nonnegative integer because 
\(\mathbb{N}\) is dense in \(\mathbb{Z}_p\). 
Moreover, it suffices to show~(\ref{eqn:2-7}) in the case where  \(h\) is not divisible by \(p\). 
In fact, assume that~(\ref{eqn:2-7}) holds for any \(h\in\mathbb{N}\) not divisible by \(p\). 
Then, for any nonnegative integer \(h=h' p^s\), where \(s=v_p(h)\geq 1\), we see 
\begin{align*}
(1+ap^e)^{hp^k}=(1+ap^e)^{h'p^{k+s}}\equiv 1\equiv 1+ah p^{k+e} \pmod  {p^{k+e+1}}, 
\end{align*}
which implies~(\ref{eqn:2-7}). \par
First, we show~(\ref{eqn:2-7}) in in the case of \(h=1\), namely, 
\begin{align}
(1+ap^e)^{p^k}\equiv 1+a p^{k+e} \pmod  {p^{k+e+1}}.
\label{eqn:2-8}
\end{align}
If \(k=0\), then~(\ref{eqn:2-8}) is trivial. If \(k\geq 1\), then the inductive hypothesis implies that 
\[
(1+ap^e)^{p^{k-1}}=1+ap^{e+k-1}+c p^{e+k}
\]
for some integer \(c\), and so 
\begin{align*}
(1+ap^e)^{p^k}=(1+ap^{e+k-1}+c p^{e+k})^p \equiv (1+ap^{e+k-1})^p \pmod  {p^{k+e+1}}.
\end{align*}
Since 
\[
(1+ap^{e+k-1})^p=1+a p^{e+k}+\sum_{j=2}^{p}\binom{p}{j} (ap^{e+k-1})^j,
\]
we deduce~(\ref{eqn:2-8}), using 
\[e+k<p(e+k-1)\]
by \(k\geq 1\), and \(e\geq 2\) or \(p\geq 3\). \par
Finally, if \(h\geq 0\) is a general integer not divisible by \(p\), then~(\ref{eqn:2-7}) follows from~(\ref{eqn:2-8}) by considering the binomial expansion of 
\((1+a p^{k+e})^h\).
\end{proof}
\begin{proof}[Proof of Proposition~\ref{pro2}]
Let \(k\) be any positive integer and \(u,h\in \mathbb{Z}_p\). First, we assume that 
\(e\geq 2\) or \(p\geq 3\). Using Lemma~\ref{lem2-1}, we get 
\begin{align*}
g(u+hp^k)&=g(u)(1+a p^e)^{hp^k}\\
&\equiv g(u)(1+ah p^{e+k}) \pmod  {p^{k+e+1}}\\
&\equiv g(u)+hp^k\cdot a p^e \pmod  {p^{k+e+1}}
\end{align*}
by \(g(u)\equiv 1 \ (\mathrm{mod} \ p)\), which implies the first statement. \par
Next, suppose that \(e=1\) and \(p=2\). In the same way as above, using Lemma~\ref{lem2-1} again, 
we see by \(k-1\geq 0\) that 
\begin{align*}
g(u+2^k\cdot h)&=g(u) (1+a' \cdot 2^t)^{h\cdot 2^{k-1}}\\
&\equiv g(u)+(h\cdot 2^k)\cdot (a'\cdot 2^{t-1}) \pmod  {2^{k+t}},
\end{align*}
which implies the second statement. 
\end{proof}

We are now ready to give a proof of Theorem~\ref{thm2}.

We may assume that \(m\) and \(p\) are coprime. In fact, if \(m\) is not coprime to \(p\), then putting 
\(m=:m' p^s\), where \(s=v_p(m)\) and \(m'\geq 2\), we have 
\[
e_p(w;m^n)\geq e_p(w;m'^n).
\]
Put \(m^{p-1}=:1+ap^e\) and \(g(u):=(1+ap^e)^u\), 
where \(a\) and \(e\) are positive integers with \(p \nmid  a\) and \(u\in\mathbb{Z}_p\). 
If \(p=2\) and \(e=1\), then we define \(a'\) and \(t\) as in Proposition~\ref{pro2}. \par
For any finite word \(v=v_{d-1}v_{d-2}\cdots v_0\) on the alphabet \(\mathcal{A}_p\), we put 
\[\varphi_p(v):=\sum_{i=0}^{d-1} v_i p^i.\]
Moreover, for any positive integer \(L\), let 
\[\varphi_p(w^L 0^c 1)=:b_{p,L},\]
for some $c$ that we will determine later. \par
Put \(F(u):=g(u)-b_{p,L}\) for \(u\in\mathbb{Z}_p\). We apply Proposition~\ref{pro1} with \(u=0,N=0,\)
\begin{align*}
j=\begin{cases}
e & \mbox{if }e\geq 2\mbox{ or }p\geq 3,\\
t-1 & \mbox{if }e=1\mbox{ and }p=2,
\end{cases}
\end{align*}
\(s=j+1\), \(n=j+1\) and put \(c:=n-1\). Then we see 
\[v_p \big(F(0)\big)=v_p(1-b_{p,L})\geq n,\]
which implies~(\ref{eqn:2-3}). 
Moreover, the assumption on the differentiability and~(\ref{eqn:2-4}) in Proposition~\ref{pro1} 
are satisfied by Proposition~\ref{pro2}. \par
Thus, Proposition~\ref{pro1} implies that there exists \(\xi\in \mathbb{Z}_p\) satisfying \(g(\xi)=b_{p,L}\). 
Let \(L'\) be the length of the word \(w^L 0^c1\). Then we have 
\[L'= l L+c+1.\]
Let \(N\) be an integer with 
\[p^{L'}\leq N<2 p^{L'}\]
and 
\[N\equiv \xi \pmod {p^{L'}}.\]
Using~(\ref{eqn:2-1}), we get 
\begin{align}
m^{(p-1)N}=g(N)\equiv g(\xi)=b_{p,L} \pmod {p^{L'}}
\label{eqn:2-9}
\end{align}
Putting \(N'=(p-1)N\), we obtain by~(\ref{eqn:2-9}) and \(m^{N'}>p^{L'}\) that %%kaneko added \(m^{N'}>p^{L'}\)
\begin{align}
e_p(w;m^{N'})\geq \gamma(w) (L-1) %%kaneko replaced L by (L-1)
\label{eqn:2-10}
\end{align}
and that %%kaneko deleted ``by~(\ref{eqn:2-9})''
\begin{align}
\log N'&\leq \log \big(2(p-1)\big)+L'\log p\nonumber\\
&= \log \big(2(p-1)\big)+(c+1)\log p+lL\log p.
\label{eqn:2-11}
\end{align}
Combining~(\ref{eqn:2-10}) and~(\ref{eqn:2-11}), we deduce Theorem~\ref{thm2} by letting \(L\) tend to infinity. 

\section*{Acknowledgements}

The first author is supported by JSPS KAKENHI Grant Number 15K17505. The second author acknowledges the support of the bilateral project ANR-FWF (France-Austria) called MUDERA (Multiplicativity, Determinism, and Randomness), ANR-14-CE34-0009.

\end{document}